\documentclass[article,11pt,leqno]{article}
\usepackage{articlemacro}

\newcommand{\Sm}{\mathbf{Sm}}
\DeclareMathOperator{\DM}{DM}

\newcommand{\quot}[1]{\left[#1\right]}
\newcommand{\Bres}{\textup{Bar}}
\newcommand{\ft}{\textup{ft}}
\newcommand{\B}[1]{\textup{B}#1}

\newcommand{\KGL}{\textup{KGL}}
\newcommand{\SH}{\textup{SH}}
\newcommand{\KH}{\textup{KH}}
\newcommand{\Sc}{\mathbf{Sc}}

\newcommand{\Repc}{\mathbf{Rep}}
\newcommand{\iGrpd}{\infty\textup{-\textbf{Grpd}}}

\title{$T$-equivariant motives of flag varieties}

\author{Can Yaylali}

\begin{document}

\maketitle

\begin{abstract}
We use the construction of the stable homotopy category by Khan-Ravi to calculate the integral $T$-equivariant $K$-theory spectrum of a flag variety over an affine scheme, where $T$ is a split torus associated to the flag variety. More precisely, we show that the $T$-equivariant $K$-theory ring spectrum of a flag variety is decomposed into a direct sum of $K$-theory spectra of the classifying stack $\B{T}$ indexed by the associated Weyl group. We also explain how to relate these results to the motivic world and deduce classical results for $T$-equivariant intersection theory and $K$-theory of flag varieties.\par For this purpose, we analyze the motive of schemes stratified by affine spaces with group action, that preserves these stratifications. We work with cohomology theories, that satisfy certain vanishing conditions, which are satisfied for example by motivic cohomology and $K$-Theory.
\end{abstract}

\tableofcontents

\section{Introduction}

\subsection*{Motivation}
Let $G$ be a split reductive group over a field $k$ with split maximal torus $T$ contained in a Borel subgroup $B$ of $G$. The geometry of the flag variety $G/B$ plays an important role in representation theory and the Langlands program. One of the aspects is to analyze the $T$-equivariant cycles of Schubert cells. There are various results on the $T$-equivariant intersection theory of a flag variety (cf. \cite{Brion}) or even on the $T$-equivariant $K_{0}$ of it (cf. \cite{Uma}). For example, $A^{\bullet}_{T}(G/B)$ has an $A^{\bullet}_{T}(k)$-basis given by precisely the classes of the Schubert cells. Analogously the same is true for $K^{T}_{0}(G/B)$, i.e. the classes of the Schubert cells yield an $R(T)$-basis. There is no canonical way to imply the former via the latter, as the equivariant Chern character map fails to be an isomorphism without completion along the augmentation ideal (cf. \cite[Thm. 1.2]{KriRR}). But this result shows that after completion this was to be expected, also for higher $K$-theory.\par 
One idea to bypass this problem is via passage to the motivic theory, which implies both results simultaneously.
Motives were famously envisioned by Grothendieck. For any variety $X$ one should be able to encode the analogous behaviors of cohomology theories in an abelian category, the category of motives of $X$. In our context, the similar  behavior of $T$-equivariant $K$-theory and intersection theory of a flag variety should be seen motivically. This is the starting point of this article.

\subsection*{Some motivic background}

Defining a suitable abelian category of motives is a difficult task. One approach that has been studied over the years is to define the derived category of motives directly and attach a $t$-structure that recovers the (abelian) category of motives as the heart of this $t$-structure. There are many constructions of the derived category of motives by Voevodsky, Morel, Ayoub, Cisinki, D\'eglise, Spitzweck and more. Under certain assumptions, it is shown that the various definitions of the derived category of motives agree. Also, in the recent constructions, the category of motives comes equipped with a full $6$-functor formalism which has become a powerful tool in analyzing functorial properties of cohomology theories.\par  
 We will in particular follow the construction of Voevodsky-Morel (cf. \cite{MV1}). Roughly, for a scheme $S$ they define the stable homotopy category $\SH(S)$ as the category of simplicial Nisnevich sheaves over $\textup{Sm}_{S}$ with coefficients in $\ZZ$, where for any smooth scheme $X$ over $S$ one inverts the structure map $\AA^{1}_{X}\rightarrow X$ (resp. the map induced on the associated representable sheaves) and  inverts `tensoring' with $\PP_{S}^{1}$ (on simplicial Nisnevich sheaves there is a closed monoidal structure given by the smash product, cf. \textit{op.cit.}). Ayoub has shown in his thesis that the association $S\mapsto \SH(S)$ defines a functor that supports a full $6$-functor formalism, i.e. $\SH(S)$ is closed monoidal and for finite type morphisms $f\colon S'\rightarrow S$ there exist adjunctions of functors $f_{!}\dashv f^{!}$, $f^{*}\dashv f_{*}$ between $\SH(S')$ and $\SH(S)$ with various compatibilities (cf. \cite[Scholie 1.4.2]{ayoub}). If $S$ is regular over a field and we work with \'etale sheaves with $\QQ$-coefficients (this is usually denoted by $\SH_{\QQ,\et}(S)$), this is equivalent to the construction of Cisinksi-D\'eglise (cf. \cite[5.3.35, Thm. 16.1.4]{CD1}).\par  Voevodsky and Morel show that there exists an object $\KGL\in \SH(S)$, which we call the \textit{motivic $K$-theory spectrum}, such that $$\Hom_{\SH(S)}(1_{S}[n],\KGL) \cong K_{n}(S)$$ for any $n\in\ZZ$, where $1_{S}$ denotes the $\otimes$-unit in $\SH(S)$. Further, by the work of Spitzweck - which relies on the moving lemma proved for Bloch's cycle complex by Levine (cf. \cite[Prop. 1.3]{LevineKM}), we know that there exists an object $M\ZZ(n)\in\SH(S)$ such that $$\Hom_{\SH(S)}(1_{S},M\ZZ(n)[2n])\cong A^{n}(S),$$ at least when $S$ is a smooth scheme over a field (cf. \cite[Cor. 7.19]{Spitz}). So, working with objects in the stable homotopy category enables us the deduce results in $K$-theory resp. intersection theory. In this way, one can also define a derived category of motives as modules over a chosen ring object in $\SH(S)$. For example Spitzweck defines a derived category of motives $\DM(S)$ as $\SH(S)$-modules over $M\ZZ$ and Cisinksi-D\'eglise define $\DM(S,\QQ)$ as $\SH(S,\QQ)$-modules over $M\ZZ\otimes_{\ZZ}\QQ$ (here $\SH(S,\QQ)$ is the stable homotopy category associated to sheaves with rational coefficients). Thus, by understanding $\SH(S)$ resp. the corresponding sheaves represented by smooth $S$-schemes, we can understand their cohomological behavior and their behavior in motivic categories. \par
To further generalize these constructions to the equivariant setting one needs to be careful. Usually, this is done by working with quotient stacks and imposing \'etale descent on $\SH$. The idea then is to work with quotient stacks that can be smoothly covered by schemes and glue the corresponding motivic categories along the atlas.
A drawback of the gluing process is that we lose information on the $K$-theoretic side. By the works of Carlsson-Joshua and Khan-Ravi, one is able to see that the glued motivic $K$-theory spectrum does not represent genuine equivariant $K$-theory but rather its completed version (cf. \cite[Thm 1.2]{CJ} and \cite[Cor 7.1]{KR2}). This is due to the fact, that on algebraic stacks (in the sense of the \cite{stacks-project}) $K$-theory does not satisfy \'etale descent. For example, for a field $k$ we have $K_{\Gm}(\Spec(k)) = \ZZ[t,t^{-1}]$ (the $\Gm$-equivariant $K$-theory of $\Spec(k)$) but the simplicial limit along the smooth cover $\Spec(k)\rightarrow \B{\Gm}$ yields
$$
	\lim_{[n]\in \Delta}K(\Gm^{n}) = \ZZ[\![T]\!].
$$
One can compute this limit as the completion of $\ZZ[t,t^{-1}]$ along $(1-t)$, which is the ideal generated by the virtual rank $0$ bundles in $K_{\Gm}(\Spec(k))$.

\subsection*{Motives of flag varieties}
In this article we are interested in the $T$-equivariant motive of the flag variety $G/B$. Before we start to explain our results, let us recall some results in the non-equivariant setting. The scheme $G/B$ classifies complete flags on an $n$-dimensional representation $V$ of $G$, i.e. subspaces $V_{0}\subseteq V_{1}\subseteq \dots\subseteq V_{n}$, where each $V_{i}$ has dimension $i$. We will recall in Section \ref{sec.Bruhat} a well-known fact, that $G/B$ admits a finite stratification by affine spaces indexed by the Weyl group $W$ of $T$ in $G$. The strata are called Schubert cells, usually denoted by $C_{w}$ for a $w\in W$, and the closure of $C_{w}$ is called Schubert variety, usually denoted by $X_{w}$. The cohomology of these objects plays an important role in representation theory and enumerative geometry. In the former, the intersection cohomology of Schubert varieties can be related to so-called Kazhdan-Lusztig polynomials (cf. \cite{KL}). For the latter, one can see that the intersection ring of $G/B$ has a basis consisting of Schubert cells (cf. \cite[\S 14]{Fulton}). Fulton gives examples on how this result together with the multiplicative structure on Chow rings leads to solutions of enumerative problems  - this is also known as Schubert calculus (cf. \cite[\S 14.7]{Fulton}, \cite{Schubert}). More generally, instead of the intersection ring, one can do Schubert calculus on the Grothendieck ring of $G/B$ (cf. \cite{BrionLec} for an exposition).\par 
In recent years the use of $\AA^{1}$-homotopy theory became more apparent in enumerative geometry (e.g. \cite{LevPaul}). In that regard, one could ask if we can obtain a Schubert calculus on other $\AA^{1}$-cohomology theories. In fact, Hornbostel-Kiritchenko show that on algebraic cobordism, we can find a similar Schubert calculus as in the Chow group case (cf. \cite{Horn}). \par
As hinted in \cite[\S 4]{BrionLec}, such results are also interesting in the $T$-equivariant setting. At least for the $T$-equivariant Grothendieck ring, this was studied by Griffeth-Ram (cf. \cite{Griff}). In \textit{op.cit.} there is also a positivity conjecture (cf. \cite[Conj. 4.1]{Griff}), which was proven later by Anderson-Griffeth-Miller (cf. \cite[Cor. 5.1]{AGM}) generalizing the result of Graham on the Schubert calculus on the equivariant cohomology ring (cf. \cite[Cor. 4.1]{Graham}).

\subsection*{Back to our setting}
As we are interested in equivariant $K$-theory as an $R(T)$-module, where $R(T) = K^{T}_{0}(\Spec(k))$ is the representation ring of $T$, it suffices to look at our problem over the classifying stack of $T$, i.e. work with $p\colon \quot{T\bs G/B}\rightarrow\B{T}$. The benefit of this viewpoint is that we only have to deal with representable maps. Further, as $T$ is a split maximal torus, we know that the derived category of $\B{T}$ with quasi-coherent cohomology is compactly generated. This fact allows for a genuine construction of the stable homotopy category $\SH$ with a six-functor formalism on $\B{T}$ (cf. \cite[Thm. 1.1]{Hoy1}). This can be extended to relatively representable algebraic stacks over $\B{T}$ (cf. \cite{KR1}). In \textit{op.cit.} Khan-Ravi show, using the works of Hoyois, that there is a $K$-theory spectrum in $\SH(\B{T})$ that represents genuine equivariant $K$-theory.\par 
Using this version of the stable motivic homotopy category, we can analyze the `motive' of $ \quot{T\bs G/B}$ relative to $\B{T}$ with coefficients in an $E_{\infty}$-ring spectrum $M_{\B{T}}\in \SH(\B{T})$, i.e. $p_{!}p^{!}M_{\B{T}}$. Because of technical reasons, we have to assume that $M_{\B{T}}$ satisfies some vanishing condition, namely for all $n\geq0$ we have
\begin{equation}
\tag{$\ast$}
\label{eq.intro.ass}
		\Hom_{\SH(\B{T})}(1_{\B{T}}\langle n\rangle [-1],M_{\B{T}}) = 0,
\end{equation}
here $\langle n\rangle \coloneqq (n)[2n]$ denotes the Tate-twist by $n$ and shift by $2n$. This is not a drawback, as we will see in Example \ref{ex.vanish} that (\ref{eq.intro.ass}) is satisfied for motivic cohomology and $K$-theory.
 As in the intersection theory case, we can use the Bruhat decomposition of $G/B$ to stratify the flag variety via $T$-invariant affine cells and then  compute $p_{!}p^{!}M_{\B{T}}$ using this stratification. \par 
The theory of Khan-Ravi works even in the case where our base is not a field but an affine scheme. So, we will also prove our results in the most general case we are able to. Thus, from now on let $S$ be an affine scheme and $(G,B,T)$ be defined over $S$. First, we have to give a Bruhat decomposition in this setting (even though this is probably known to many people, we didn't find a reference and proved it by ourselves).

\begin{intro-proposition}[\protect{Prop. \ref{lem.GB.cellular}}]
	Let $S$ be a non-empty scheme (not necessarily affine). Let $G$ be a split reductive $S$-group scheme with split maximal torus $T$ and a Borel subgroup $B$ containing $T$. Then the $S$-scheme $G/B$ admits a cellular stratification indexed by the Weyl group of $T$ in $G$.
\end{intro-proposition}

Afterward, we analyze the motive of a proper scheme $X$ endowed with a group action of an $S$-group scheme $H$ and an $H$-invariant cellular decomposition. In the special case of $G/B$ with $T$-action this yields the structure of $p_{!}p^{!}M_{\B{T}}$ as a $M_{\B{T}}$-module with basis given by the classes of the Schubert cells.

\begin{intro-theorem}[\protect{Cor. \ref{cor.T-Flag}}]
\label{intro-thm.1}
	Let $G$ be a split reductive $S$-group scheme with maximal split torus $T$ that is contained in a Borel subgroup $B$. Then 
	$$
		p_{!}p^{!}M_{\B{T}}\simeq \bigoplus_{w\in W} M_{\B{T}}\langle l(w)\rangle,
	$$
	where $W$ is the Weyl group of $T$ in $G$.
\end{intro-theorem}

Applying this result with the representation of homotopy invariant $K$-theory in $\SH$, we get the decomposition of homotopy invariant $K$-theory.

\begin{intro-theorem}[\protect{Cor. \ref{cor.KH.flag}}]
\label{int.thm.2}
Let $S$ be an affine scheme. Further, let $G$ be a split reductive $S$-group scheme with maximal split torus $T$ that is contained in a Borel subgroup $B$. Then we have
$$
	\KH(\quot{T\bs G/B})\simeq \bigoplus_{w\in W} \KH(\B{T})
$$
\end{intro-theorem}

On the $0$-th homotopy group we recover an integral version of the classical result, that $K^{T}_{0}(G/B)$ as an $R(T)$-module is generated by the $T$-equivariant classes of the Schubert cells. For the higher equivariant $K$-groups, we get a similar statement.

\begin{intro-corollary}[\ref{eq.higher.K}]
\label{int.cor.1}
	Let $S$ be a Noetherian regular affine scheme. Further, let $G$ be a split reductive $S$-group scheme with maximal split torus $T$ that is contained in a Borel subgroup $B$. Let $R(T)$ denote the integral representation ring of $T$. Then we have an isomorphism of $R(T)$-modules
	$$
		K_{i}^{T}(G/B)\coloneqq K_{i}(\quot{T\bs G/B}) \cong \bigoplus_{w\in W} K^{T}_{i}(S).
	$$
\end{intro-corollary}
Now let us assume that $S=\Spec(k)$ is the spectrum of a field. On the higher homotopy groups, we similarly get an isomorphism of $R(T)_{\QQ}$-modules via tensoring with the higher $K$-groups of the ground field.

\begin{intro-corollary}[\ref{eq.Ki.flag}]
\label{int.cor.2}
Let $k$ be a field. Further, let $G$ be a split reductive $k$-group scheme with maximal split torus $T$ that is contained in a Borel subgroup $B$. Then we have an isomorphism of $R(T)_{\QQ}$-modules
$$
	K^{T}_{i}(G/B)_{\QQ}\cong K_{i}(k)_{\QQ}\otimes_{\QQ} K^{T}_{0}(G/B)_{\QQ}.
$$
\end{intro-corollary}
Using the universal property of $\SH$, we can also extend these results to the lisse-extension of $\SH$ and the \'etale localized rational stable homotopy category $\SH_{\QQ,\et}$ (we have to assert some conditions on the base as seen in Section \ref{sec.gen.motive}). In this way, we can extend Theorem \ref{intro-thm.1} to the case of Beilinson motives and recover the analogous result on $A^{*}_{T}(G/B)$ and completed $K_{0}$.

\begin{intro-proposition}[\ref{eq.Chow.flag} and \ref{eq.K.flag}]
\label{int.prop.1}
Let $S=\Spec(k)$ be the spectrum of a field. Further, let $G$ be a split reductive $S$-group scheme with maximal split torus $T$ that is contained in a Borel subgroup $B$. Then on completed equivariant $K$-theory, we have
	\begin{align*}
		 K_{0}^{T}(G/B)_{\QQ}^{\wedge I_{T}}&\cong K_{0}(G/B)_{\QQ}\otimes_{\QQ} K^{T}_{0}(S)_{\QQ}^{\wedge_{I_{T}}},
	\end{align*}
where $I_{T}$ is the ideal generated by virtual rank $0$-bundles in $R(T)_{\QQ}$. \par
On Chow rings, we recover
$$
	 A^{*}_{T}(G/B)_{\QQ}\cong A^{*}(G/B)_{\QQ} \otimes_{\QQ} A_{T}^{*}(S)_{\QQ}.
$$
\end{intro-proposition}

\subsection{Notation}
\subsubsection*{Categorical Notation}
In this paper, we will without further mention use the language of $\infty$-categories (cf. \cite{HTT}). We will identify $1$-categories with their Nerve and regard them as $\infty$-categories. In particular, we will identify the category of sets with the full sub $\infty$-category of $0$-truncated $\infty$-groupoids and the category of groupoids with the full sub $\infty$-category of $1$-truncated $\infty$-groupoids. Likewise, when we say \textit{full subcategory} of an $\infty$-category, we will always mean a full sub $\infty$-category. \par 
For the rest of this article, we fix an uncountable inaccessible regular cardinal $\kappa$ and \textit{small} will mean $\kappa$-small. Without further mention, if needed, we will assume smallness of the categories involved in this article. Indexing sets will always be small. We denote by $\ICat$ the $\infty$-category of small $\infty$-categories and by $\iGrpd$ the $\infty$-category of small $\infty$-groupoids.\par 
A \textit{presheaf} $\Fcal$ on an $\infty$-category $\Ccal$ is a functor $\Fcal\colon \Ccal^{\op}\rightarrow \iGrpd$. If $\Ccal$ admits a Grothendieck topology $\tau$, we will say that $\Fcal$ is a \textit{$\tau$-sheaf} if it is a sheaf with respect to the topology $\tau$. \par 

\subsubsection*{Algebraic geometric notation}
Let $S$ be a scheme. By an algebraic stack $X$ over $S$, we mean an \'etale-sheaf of groupoids on $S$-schemes, such that the diagonal of $X$ is representable by an algebraic space and there exists a scheme $U$ and a smooth effective epimorphism $U\rightarrow X$. A morphism of algebraic stacks over $S$ will always be an $S$-morphism.

\subsection*{Structure of this article}
In the first section, we recall some facts about the stable homotopy category in our setting after Khan-Ravi. Afterward, we prove some basic facts, we need later on.\par 
Our next step is to show the existence of the Bruhat decomposition of $G/B$ over arbitrary schemes. We then continue to compute the motive of strict linear schemes with group action and apply this to $\quot{T\bs G/B}$.\par 
We conclude this article, by applying our result to integral and rational homotopy invariant $K$-theory and their homotopy groups. Finally, we discuss how one extends these results to other motivic categories and get the classical results on Chow rings.

\subsection*{Acknowledgement}
I would like to thank Torsten Wedhorn for his remarks on the earlier versions of this article and for sketching me his idea on the Bruhat decomposition. Also, I would like to thank Simon Pepin Lehalleur, for pointing out an error in the first version of this manuscript and the discussion with him afterward, that led to the corrected version. Further, I would also like to thank Timo Richarz and Thibaud van den Hove for various discussions concerning this paper. \par 
This project was funded by the Deutsche Forschungsgemeinschaft (DFG, German Research Foundation) - project number 524431573, the Deutsche Forschungsgemeinschaft (DFG, German Research Foundation) TRR 326 \textit{Geometry and Arithmetic of Uniformized Structures}, project number 444845124 and by the LOEWE grant `Uniformized Structures in Algebra and Geometry'.


\section{Motivic setup}
In this section, we fix an affine scheme $S$, a split reductive $S$-group scheme $G$ together with a Borel pair $(B,T)$ consisting of a split maximal torus $T$ inside a Borel subgroup $B$ of $G$. Further, any algebraic stack will be considered as an algebraic stack over $S$ and any morphism will be relative over $S$.\par 
In this article, we want to compute the motive of $T$-equivariant flag varieties. In particular, we are interested in motives of Artin stacks. There are several approaches on how to extend the theory of motives to Artin stacks. Recall from the introduction that one can right Kan extend $\SH_{\QQ,\et}$ from schemes to Artin stacks and get a full $6$-functor formalism. This works fine until one wants to compute the motivic cohomology in terms of K-theory. One can show that $K$-theory does not satisfy \'etale descent for Artin stacks\footnote{In fact, $G$-theory satisfies descent precisely for \'etale covers that are ``isovariant'' (cf. \cite[Thm. 1.1]{Jo1}).}. For example, let $\quot{X/G}$ be a smooth Artin stack over a field $k$ with $G$ split reductive. Then we have a map $K(\quot{X/G})\rightarrow K^{\et}(\quot{X/G})$, where $K^{\et}$ is the right Kan extended $K$-theory from algebraic spaces to \'etale sheaves. This map is not an equivalence but realizes $K^{\et}_{0}(\quot{X/G})$ as the completion of $K^{G}_{0}(X)$ along the augmentation ideal $I_{G}\subseteq R(G)= K(\textup{Rep}(G))$ (note that $K^{G}_{0}(X)$ is in general not $I_{G}$-complete as seen for the $\Gm$-equivariant $K$-theory of a point, c.f. Example \ref{ex.K.comp}). This is an instance of the comparison between the Borel construction for $K$-theory and equivariant $K$-theory (cf. \cite[Thm. 1.3]{Kri}). \par
Also for non-rational coefficients, one has to be careful as \'etale descent is not even satisfied for schemes. Hence, one has to be careful to construct a full $6$-functor formalism for $\SH$ with non-rational coefficients. This was done by Chowdhury in his thesis by gluing along smooth morphism with Nisnevich local sections (cf. \cite[Thm. 1.0.1]{Chowd}). This is equivalent to gluing along smooth covers, the so-called \textit{lisse-extension} (cf. \cite[Cor. 12.28]{KR1}). But again, computing the motivic cohomology along the lisse-extended $K$-theory spectrum yields the completion of $K$-theory along the augmentation ideal (cf. \cite[Ex. 12.22]{KR1}).\par 
In the case of $\Xcal\coloneqq \quot{T\bs G/B}$ there is a construction by Khan-Ravi of a stable homotopy category $\SH(\Xcal)$  that admits a full $6$-functor formalism and a motivic spectrum $\KGL_{\Xcal}\in\SH(\Xcal)$ such that 
$$
	\Hom_{\SH(\Xcal)}(1_{\Xcal},\KGL_{\Xcal}) = \KH(\Xcal),
$$
where $\KH(\Xcal)$ denotes the homotopy invariant $K$-theory of $\Xcal$ (cf. \cite[Const. 10.3]{KR1}). The quotient stack $\Xcal$ belongs to a certain class of algebraic stacks, called \textit{scalloped} (see below) for which the stable homotopy category is also defined.\par 
In the end of this article, we will look at the implications on motivic cohomology in various frameworks (cf. Section \ref{sec.gen.motive}).  

\subsection{Scalloped stacks}
We recall the necessary definitions from \cite[\S 2]{KR1}. We will use the terminology of \textit{loc.cit.}.

\begin{defi}
Let $H$ be a group scheme over $S$. We say that $H$ is \textit{nice} if it is an extension of an \'etale group scheme of order prime to the residue characteristics of $S$, by a group scheme of multiplicative type.
\end{defi}

\begin{example}[cf. \protect{\cite[Rem. 2.2]{AHR1}}]
	An important example of a nice group scheme is a torus. One can show that any nice group scheme is linearly reductive. If $S$ is the spectrum of a field of characteristic $p>0$, then linearly reductive group schemes are also nice.
\end{example}

Let $H$ be a nice $S$-group scheme and $X$ a quasi-affine scheme with action by $H$. Hoyois constructs an equivariant version of the stable homotopy category $\SH^{H}(X)$ with full $6$-functor formalism in this context (cf. \cite[Thm. 1.1]{Hoy1}). Khan and Ravi extend this construction via gluing along Nisnevich squares to a class of algebraic stacks, which they call \textit{scalloped}. We don't want to give an explicit definition of a scalloped stack, as it is a bit technical, but give an important example and some properties of scalloped stacks, for the definition and details we refer to \cite[\S 2]{KR1}. 

\begin{prop}[\protect{\cite[Cor 2.13, Thm. 2.14]{KR1}}]
$ $
\begin{enumerate}
	\item[(i)] Let $f\colon X'\rightarrow X$ be a morphism of qcqs algebraic stacks. If $X$ is scalloped and $f$ is representable, then $X'$ is scalloped.
	\item[(ii)] Let $X$ be a qcqs algebraic space over $S$ with $H$-action, where $H$ is a nice $S$-group scheme. Then $\quot{X/H}$ is scalloped.
\end{enumerate}
\end{prop}

	Throughout this article, we are interested in quotients of qcqs schemes by the torus $T$, such as $\quot{T\bs G/B}$. The above proposition tells us that these stacks are scalloped. In particular, we can work with the formalism of \cite{KR1}.
	
\begin{notation}
		We set $\Repc_{S}$ to be the $\infty$-category of morphisms $X'\rightarrow X$ of algebraic stacks over $S$ that are representable. We denote by $\Repc_{S}^{\ft}$ the full subcategory of $\Repc_{S}$ consisting of morphisms of finite type over $S$. Further, we denote by $\Sc_{S}$ resp. $\Sc_{S}^{\ft}$ the full subcategories of $\Repc_{S}$ resp. $\Repc_{S}^{\ft}$ consisting of scalloped stacks.
\end{notation}

\subsection{The stable homotopy category for scalloped stacks}
Let us quickly recall the construction of $\SH$ for scalloped stacks from \cite{KR1}. For a scalloped stack $X$ let us set $\Sm_{X}$ as the full subcategory of $(\Repc_{S})_{/X}$ consisting of morphisms $X'\rightarrow X$ that are smooth and representable.\par
 
We define the homotopy category $\textup{H}(X)$ of $X$, as the $\infty$-category of Nisnevich sheaves $F$ from $\Sm_{X}$ to $\iGrpd$ that are homotopy invariant, i.e. for any $X'\in \Sm_{X}$ and any vector bundle $p\colon V\rightarrow X'$, we have that the induced map $F(X')\rightarrow F(V)$ is an equivalence. In the classical motivic theory, for example of Cisinski-D\'eglise, one obtains the stable homotopy category by adjoining $\otimes$-inverses of Thom-motives of finite locally free sheaves. In our case, we can associate to any finite locally free module $\Ecal$ over $X$ an object $\langle\Ecal\rangle\in \textup{H}(X)$, called the \textit{Thom-anima} (cf. \cite[\S 4]{KR1}). Now we obtain the stable homotopy category, by formally $\otimes$-inverting these Thom-anima. One should note that formal $\otimes$-inversion of objects in $\infty$-categories is more delicate and we refer to \textit{op.cit.} for references and details.

\begin{defi}[\cite{KR1}]
	Let $X$ be a scalloped algebraic stack. The \textit{stable homotopy category of $X$} is defined as the $\infty$-category
	$$
		\SH(X)\coloneqq \textup{H}(X)[\langle\Ecal\rangle^{\otimes -1}],
	$$
	where $[\langle\Ecal\rangle^{\otimes -1}]$ denotes the formal inversion of all Thom-anima associated to any finite locally free module $\Ecal$ over $X$.
\end{defi}

The most important feature of the stable homotopy category for us is that the assignment $X\mapsto \SH(X)$ can be upgraded to a functor with a full $6$-functor formalism, that satisfies homotopy invariance and yields a localization sequence. Further, homotopy invariant $K$-theory resp. motivic cohomology of $X$ can be represented by objects in $\SH(X)$.\par 
Let us quickly recall the $6$-functor formalism for scalloped stacks on the stable homotopy category $\SH$. We also recall the comparison with $K$-theory.

\begin{thm}[\protect{\cite{KR1}}]
\label{thm.6.ff}
	For any scalloped stack $X$ there is an $\infty$-category $\SH(X)$ with the following properties
\begin{enumerate}
	\item[(i)] $\SH(X)$ is a stable, presentable, closed symmetric monoidal $\infty$-category. The tensor product is colimit preserving and the inner $\Hom$ will be denoted by $\Homline$. The $\otimes$-unit will be denoted by $1_{X}$.
	\item[(ii)] The assignment $X\mapsto \SH(X)$ can be upgraded to a presheaf of symmetric monoidal presentable $\infty$-categories with colimit preserving functors on the site of scalloped stacks
	$$
		\SH^{*}\colon (\Sc_{S})^{\op}\rightarrow \ICat^{\otimes},\ X\mapsto \SH(X),\ f\mapsto f^{*}.
	$$
	  For any morphism $f\colon X\rightarrow Y\in\Sc_{S}$, there is an adjunction
	$$
		\begin{tikzcd}
			 f^{*}\colon \SH(Y)\arrow[r,"",shift left = 0.3em]&\arrow[l,"",shift left = 0.3em]\SH(X)\colon f_{*}.
		\end{tikzcd}
	$$
	\item[(iii)] (Homotopy invariance) For every vector bundle $p\colon V\rightarrow X$ of scalloped stacks, the unit of the $*$-adjunction
	$$
		1\rightarrow p_{*}p^{*}
	$$
	is an equivalence.
	\item[(iv)] If $f\in\Sc_{S}$ is smooth morphism, then $f^{*}$ has a left adjoint, denoted $f_\sharp$ that is a morphism of $\SH(Y)$-modules.
	\item[(v)] The assignment $X \mapsto \SH(X)$ can be upgraded to a presheaf of presentable $\infty$-categories
$$\SH: (\Sc_{S}^{\ft})^{\op} \rightarrow \ICat,\ X \mapsto \SH(X),\ f \mapsto f^!$$ from the $\infty$-category of scalloped stacks with finite type representable morphisms.
For each $f\colon X\rightarrow Y$ in $\Sc_{S}^{\ft}$, there is an adjunction
$$f_! : \SH(X) \rightleftarrows \SH(Y): f^!.$$
For any factorization $f = p \circ j$ with $j$ an open immersion and $p$ a proper representable map, there is a natural equivalence $f_! \cong p_* j_\sharp$.
	\item[(vi)] (Localization) If $i\colon Z\rightarrow X$ is a closed immersion of scalloped stacks with open complement $j\colon U\rightarrow X$, then we have the following cofiber sequences
	\begin{align*}
		j_{\sharp}j^{*}\rightarrow \id \rightarrow i_{*}i^{*}\\
		i_{!}i^{!}\rightarrow \id \rightarrow j_{*}j^{*}.
	\end{align*}
	\item[(vii)] There is a map $K(X)\rightarrow \Aut(\SH(X))$, assigning for any $\alpha \in K(X)$ its twist $\langle \alpha \rangle$. If $\alpha$ is given by a finite locally free sheaf $\Ecal$, then $\langle \Ecal \rangle = p_{\sharp}s_{*}1_{X}$ (this agrees with the previously considered Thom-anima), where $p\colon V(\Ecal)\rightarrow X$ is the projection of the associated vector bundle and $s$ its zero section. Further, any of the $6$-operations commute with $\langle \alpha\rangle$ in a suitable sense (cf. \cite[Rem. 7.2]{KR1}). We set $\langle n\rangle \coloneqq \langle \Ocal_{X}^{n}\rangle$.
	\item[(viii)] The canonical projection $p\colon \Gm\times X\rightarrow X$ yields a morphism $p_{\sharp}p^{*}1_{X}[-1]\rightarrow 1_{X}[-1]$ whose fiber we denote $1_{X}(1)$. For an $M\in\SH(X)$, we denote its $n$-Tate twist by $M(n)\coloneqq M\otimes 1_{X}(1)^{\otimes n}$. We have $\langle n\rangle \simeq (n)[2n]$. 
	\item[(ix)] (Purity) Let $f$ be a smooth representable morphism of scalloped algebraic stacks with cotangent complex $L_{f}$, then 
		$$
			f^{!}\simeq f^{*}\langle L_{f}\rangle.
		$$
	\item[(x)] For a cartesian diagram in $\Sc_{S}^{ft}$
		$$
			\begin{tikzcd}
				W\arrow[r,"g'"]\arrow[d,"f'"]& X\arrow[d,"f"]\\
				Y\arrow[r,"g"]&Z
			\end{tikzcd}
		$$
		with $g\in\Sc_{S}^{\ft}$, we have 
		\begin{align*}
			g^! f_* & \xrightarrow{\simeq} f'_* g'^!, \\
			f^* g_! & \xrightarrow{\simeq} g'_! f'^*.
		\end{align*}
		\item[(xi)]
			For $f\colon Z\rightarrow Y$ in $\Sc_{S}^{\ft}$, the functor $f_{!}$ satisfies the projection formulas (cf. \cite[Thm. 7.1]{KR1}).
		\item[(xii)] There exists an $E_{\infty}$-ring spectrum $\KGL_{X}\in \SH(X)$ such that 
		$$
			\KH(X) =\Homline_{\SH(X)}(1_{X},\KGL_{X}),
		$$
		that is functorial in smooth representable morphisms and satisfies Bott periodicity for twist by finite locally free sheaves (cf. \cite[Thm. 10.7]{KR1}).
\end{enumerate}
\end{thm}

In the rest of this article, we want to focus on modules over an $E_{\infty}$-ring spectrum $M\in\SH(X)$, where $X$ is scalloped. The reason is that we will need a vanishing assumption (see below) that is satisfied for example for the homotopy invariant $K$-theory spectrum. As we are interested in $T$-equivariant $K$-theory of the flag variety, this is not a strong restriction. \par 
For oriented cohomology theories, we can relax our situation from flag varieties to linear schemes. To be more precise, the flag variety is stratified by affine spaces. If we are interested in oriented cohomology theories, it is enough to consider objects that are stratified by vector bundles. The next example will show how this idea works.\par 
But before we come to the example let us fix some notation.

\begin{notation}
	Let $X$ be a scalloped algebraic stack and $M_{X}\in\SH(X)$ an $E_{\infty}$-ring spectrum. Then we denote the $\infty$-category of $M_{X}$-modules in $\SH(X)$ by $\SH(X)_{M}$. Further, for any representable morphism $f\colon Y\rightarrow X$  of scalloped algebraic stacks, we denote $M_{Y}\coloneqq f^{*}M_{X}$.
\end{notation}

\begin{rem}
\label{rem.6.ff.mod}
Let $f\colon X\rightarrow Y$ be a representable morphism of scalloped algebraic stacks. Further, let $M_{Y}\in \SH(Y)$ be an $E_{\infty}$-ring spectrum. Tensoring with $M_{Y}$ (resp. $M_{X}$) induces a pullback functor $f^{*}_{M}\colon \SH(Y)_{M}\rightarrow \SH(X)_{M}$. As the $*$-pullback for $\SH$ is monoidal, we see that its right adjoint is lax-monoidal. In particular, we get an adjunction 
$$
	 \begin{tikzcd}
		f^{*}_{M}\colon \SH(Y)_{M}\arrow[r,"",shift left = 0.3em]&\arrow[l,"",shift left = 0.3em]\SH(X)_{M}\colon f_{*M}.
	\end{tikzcd}
$$
As remarked in Theorem \ref{thm.6.ff} (iv), if $f$ is smooth, the left adjoint $f_{\sharp}$ of $f^{*}$ is a morphism of $\SH(Y)$-modules and in particular, induces a left adjoint $f_{\sharp M}$ of $f^{*}_{M}$ (cf. \cite[\S 7.2]{CD1}). \par
	By conservativity of the forgetful-functor $\SH_{M}\rightarrow \SH$ (here we see $\SH_{M}$ as a functor from scalloped algebraic stacks with representable morphisms to symmetric monoidal presentable $\infty$-categories), we can use \cite[Thm 7.1]{KR1} to obtain a $6$-functor formalism on $\SH_{M}$ that satisfies properties (i)-(xi) of Theorem \ref{thm.6.ff}.
\end{rem}

\begin{notation}
	In the rest of this article, we will work with module spectra over some fixed $E_{\infty}$-ring spectrum. Thus, we will drop the subscript in the $6$-functor formalism indicating the fixed $E_{\infty}$-ring spectrum as seen in Remark \ref{rem.6.ff.mod}.
\end{notation}

\begin{defi}
	Let $f\colon X\rightarrow Y$ be a representable morphism of finite type of scalloped algebraic stacks. Further, let $M_{Y}\in \SH(Y)$ be an $E_{\infty}$-ring spectrum in $\SH(Y)$. Then we define \textit{motive of $X$ with values in $M_{Y}$} (resp. the \textit{compactly supported motive of $X$ with values in $M_{Y}$}) as $M_{Y}(X)\coloneqq f_{!}f^{!}M_{Y}$ (resp. $M_{Y}^{c}(X)\coloneqq f_{*}f^{!}M_{Y}$) in $\SH(Y)_{M}$. 
\end{defi}

\begin{example}
\label{ex.motive.of.vb}
	Let $\Ecal$ be a finite locally free sheaf on a scalloped stack $X$ and let $p\colon V(\Ecal)\rightarrow X$ be the associated vector bundle. Further, let $M_{X}\in\SH(X)$ be an $E_{\infty}$-ring spectrum that admits an orientation, i.e. there is a functorial equivalence $M_{X}\langle \Fcal\rangle \simeq M_{X}\langle n\rangle$ for any finite locally free $\Ocal_{X}$-module $\Fcal$ of rank $n$. Note that by homotopy invariance, we have that $p_{*}p^{*}\simeq \id\simeq p_{\sharp}p^{*}\in \SH(X)_{M}$. In particular, by purity we have that $M_{X}(V(\Ecal))$ and $M^{c}_{X}(V(\Ecal))$ are equivalent to $M_{X}\langle \Ecal\rangle$. Further, as we can orient the unit and by the Mayer-Vietoris sequence\footnote{Note that $f_{*}f^{*}$ and $f_{!}f^{!}$ satisfy Nisnevich descent and thus yield a Mayer-Vietoris sequence for $M_{X}$ and $M_{X}^{c}$, i.e. for open substacks $U,U'\subseteq V(\Ecal)$, we have a fiber sequence of the form
	$$
		M_{X}(U\cap U')\rightarrow M_{X}(U)\oplus M_{X}(U')\rightarrow M_{X}(V)
	$$ and similarly, for $M_{X}^{c}$} we have that $M_{X}(V(\Ecal))\simeq\bigoplus_{n\in \ZZ}M_{X}\langle \Ecal_{n}\rangle$, where the $\Ecal_{n}$ are the restrictions of $\Ecal$ to the open and closed locus $X_{n}$ such that $\Ecal_{|X_{n}}$ has rank $n$. Further, as $M_{X}$ admits an orientation, we have  $M_{X}\langle \Ecal_{n}\rangle \simeq M_{X}\langle n \rangle$.\par
	Now let $X$ be a qcqs algebraic space over $S$ and $\Ecal$ be a finite locally free sheaf on $X$. Further, let $G$ be a nice $S$-group scheme acting on $X$ such that $\Ecal$ is $G$-equivariant. This yields a vector bundle over $\quot{X/G}$, i.e. the there is a finite locally free sheaf $\Ecal_{G}$ on $\quot{X/G}$ such that $\quot{V(\Ecal)/G}\cong V(\Ecal_{G})$. As $X$ is quasi-compact, we can find a finite set $I_{\Ecal}\subseteq \NN_{0}$ and an open and closed cover $(X_{i})_{i\in I_{\Ecal}}$ of $X$ such that for any $i\in I_{\Ecal}$ the sheaf $\Ecal_{|X_{i}}$ is finite locally free of rank $i$ and $\Ecal=\bigoplus_{i\in I_{\Ecal}} \Ecal_{i}$. As the trivial bundle $X\rightarrow \quot{X/G}$ yields an atlas, we see that the same holds true for $\Ecal_{G}$. In particular, as we have seen above this yields $M_{\quot{X/G}}(V(\Ecal_{G}))\simeq M_{\quot{X/G}}\langle \Ecal_{G}\rangle = \bigoplus_{i\in I_{\Ecal}}M_{\quot{X/G}}\langle i\rangle$.
\end{example}

\begin{lem}[Localization sequence]
\label{lem.loc.for.stacks}
	Let $f\colon X\rightarrow Y$ be a representable morphism of scalloped algebraic stacks over $S$ of finite type. Further, let $M_{Y}\in \SH(Y)$ be an $E_{\infty}$-ring spectrum. Let $i\colon Z\hookrightarrow X$ be a closed immersion over $Y$ with open complement $j\colon U\rightarrow  X$. Further, let us denote $f_{0}\coloneqq f\circ j$ and $\fbar\coloneqq f\circ i$. Then for any $M_{X}$-module $N$ in $\SH(X)$ there exists the following fiber sequence in $\SH(Y)_{M}$
	$$
		\fbar_{*}\fbar^{!}N\rightarrow f_{*}f^{!}N\rightarrow f_{0*}f_{0}^{!}N.
	$$
\end{lem}
\begin{proof}
	 Applying the localization sequence 
	$$
		i_{*}i^{!}=i_{!}i^{!}\rightarrow \id\rightarrow j_{*}j^{*}=j_{*}j^{!}
	$$
	to $f^{!}N$ yields
	$$
		i_{*}\fbar^{!}N\rightarrow f^{!}N\rightarrow j_{*}f_{0}^{!}N.
	$$
	Now applying $f_{*}$ to this sequence yields the result.
\end{proof}

We will need the following vanishing assumption later on in this article. In motivic cohomology, this is the analog of the vanishing of negative higher Chow groups. For the $K$-theory spectrum, this will follow from the vanishing of negative $K$-groups.

\begin{assumption}
	\label{ass.vanish}
	Let $X=\B{H}$ be the classifying stack of a nice group scheme $H$ and $M_{X}\in \SH(X)$ be an $E_{\infty}$-ring spectrum. Further, let $n> 0$, then we have
	$$
		\Hom_{\SH(X)}(1_{X}\langle n\rangle [-1],M_{X}) = 0.
	$$
\end{assumption}

Assumption \ref{ass.vanish} is satisfied at least for the two cohomology theories, that are considered in this article. Namely, homotopy invariant $K$-theory and motivic cohomology.

\begin{example}
\label{ex.vanish}
Assume $S$ is Noetherian. Let $m<0$ and $n\in \ZZ$ and $X=\B{H}$ the classifying stack for a nice group scheme $H$. Further, let us consider the $K$-theory spectrum $\KGL_{X}\in \SH(X)$. Then we have 
	$$
		\Hom_{\SH(X)_{\KGL}}(\KGL_{X}\langle n\rangle [m],\KGL_{X}) \simeq \Hom_{\SH(X)_{\KGL}}(\KGL_{X}[m],\KGL_{X}) \simeq\pi_{m}\KH(X),
	$$
	where the first equivalence follows from Bott-periodicity (cf. Theorem \ref{thm.6.ff} (xii)).
	As $H$ is nice, the spectrum $\KH(X)$ is connective\footnote{Here we use that the derived category of complexes with quasi-coherent cohomology $D_{\textup{qc}}(X)$ is compactly generated, as $H$ is nice (combine \cite[Prop. 6.14]{AHR1} and \cite[Rem. 12.2]{Alper}). Then connectivity follows from \cite[Thm. 5.7]{HoyKri} (here we need that $S$ is Noetherian).}. In particular, since $m< 0$, we see that $\KGL_{X}$ satisfies Assumption \ref{ass.vanish}.\par
	Now let us assume that $S=\Spec(k)$ is the spectrum of a field and let us consider the motivic cohomology spectrum $M\ZZ\in \SH(X)$ (cf. \cite[Const. 10.16]{KR1}). Then
	$$
		\Hom_{\SH(X)_{M\ZZ}}(M\ZZ\langle n\rangle [m],M\ZZ) = \Hom_{\SH(X)}(1_{X},M\ZZ\langle -n\rangle [-m]) = A^{-n}(X,m),
	$$
	which vanishes as $m$ is negative.
\end{example}

The above example shows that the motivic $K$-theory spectrum and the motivic cohomology spectrum satisfy an even stronger condition than Assumption \ref{ass.vanish}. Let us give this stronger assumption a number.

\begin{assumption}
	\label{ass.vanish.strong}
	Let $X=\B{H}$ be the classifying stack of a nice group scheme $H$ and $M_{X}\in \SH(X)$ be an $E_{\infty}$-ring spectrum. Further, let $n\in\ZZ $ and $m<0$, then we have
	$$
		\Hom_{\SH(X)}(1_{X}\langle n\rangle [m],M_{X}) = 0.
	$$
\end{assumption}

\section{$T$-equivariant motivic homotopy theory of flag varieties}

Let $S$ be a non-empty scheme, $G$ be a split reductive $S$-group scheme and $T$ a maximal split torus contained in a Borel subgroup $B\subseteq G$. We want to understand the motive of $\quot{T\bs G/B}$. In this case, the computations are rather straightforward, as $G/B$ is cellular, i.e. has a stratification by vector bundles. So, one can filter $G/B$ by $T$-invariant closed subschemes such that their successive differences are given by vector bundles (usually this property is called \textit{linear} in the literature). The existence of such a stratification is well-known and referred to as the Bruhat decomposition of $G/B$. As we have only found references for split reductive groups over a field, we first recall the existence of an affine cell decomposition of $G/B$ over $S$.\par 
Afterward, we can analyze the motive of the scalloped stack $\quot{T\bs G/B}$ over $\B{T}$. Note however, that the $6$-functor formalism of $\SH$ only works for representable morphisms of scalloped stacks, so we cannot compute the motive of $\quot{T\bs G/B}$ over $S$. We will however explain in Section \ref{sec.gen.motive} how this extends to the lisse-extension of $\SH$ and Beilinson motives $\DM_{\QQ}$, which both have $6$-functor formalism for non-representable morphisms.

\subsection{Affine cell decomposition of $G/B$}
\label{sec.Bruhat}

In this section, we will show that $G/B$ has an affine cell decomposition. We will use the Bruhat decomposition of $G$ and pull back the induced stratification on $\quot{B\bs G/B}$ to $G/B$. This construction is compatible with base change and thus, we will reduce to the case $S=\Spec(\ZZ)$. Then this is a classical statement on Schubert cells. This was communicated to us by Torsten Wedhorn.

\begin{defi}
	A \textit{stratification} of a scheme $X$ is a map $\iota\colon \coprod_{i\in I} X_{i}\rightarrow X$, where $I$ is a set, each $X_{i}$ is a scheme, $\iota$ is a bijection on the underlying topological spaces, $\iota_{|X_{i}}$ is an immersion and the topological closure of $\iota(X_{i})$ in $X$ is the union of subsets of the form $\iota(X_{j})$. The subschemes $\iota(X_{i})$ of $X$ are called \textit{strata}.
\end{defi}

\begin{defi}
	An \textit{$S$-cell} is an $S$-scheme isomorphic to a vector bundle. A \textit{cellular $S$-scheme $X$} is a separated $S$-scheme of finite type which is smooth and admits a stratification whose strata are cells. 
\end{defi}

First, let us recall that $G$ admits a Bruhat-decomposition indexed by the Weyl group $W$ of a split maximal torus $T$ inside $G$.

\begin{lem}[Bruhat-decomposition]
	Let $G$ be a split reductive $S$-group scheme and $B$ a Borel subgroup of $G$ containing a split maximal torus. Then $G$ admits a stratification $\coprod_{w\in W}BwB\rightarrow G$, where $W$ denotes the Weyl group of $T$ in $W$.
\end{lem}
\begin{proof}
	See \cite[Exp. XXII, Thm. 5.7.4]{SGA3}.
\end{proof}

The Bruhat-decomposition yields a stratification $\coprod_{w\in W}\quot{B\bs BwB/B}\rightarrow \quot{B\bs G/B}$. The stack $\quot{B\bs G/B}$ can be identified with the quotient $G/B \times^{G} G/B$, where $G$ acts diagonal via conjugation. For any $S$-scheme $Q$ the set $G/B(Q)$ is in bijection with Borel-subschemes of $G_{Q}$ (cf. \cite[Exp. XXII, Cor. 5.8.3]{SGA3}). Thus, we have a map
$$
	G/B\rightarrow \quot{B\bs G/B}
$$
given on $Q$-valued points by $B'\mapsto (B',B)$. The pullback of the above stratification on $\quot{B\bs G/B}$ via this map yields a stratification $\coprod_{w\in W} C_{w}\rightarrow G/B$. We call the $C_{w}$ \textit{Schubert cells of $G/B$}.\par 
On $W$ we have a length function, which we denote by $l$ (cf. \cite[Ch. IV \S 1.1]{Bourbaki}). Then we claim that $C_{w} \cong \AA^{l(w)}_{S}$. In particular, $G/B$ has a cellular stratification.

\begin{prop}
\label{lem.GB.cellular}
	Let $S$ be a non-empty scheme. Let $G$ be a split reductive $S$-group scheme with split maximal torus $T$ and a Borel subgroup $B$ containing $T$. Then the Schubert cell $C_{w}$ is isomorphic to $\AA^{l(w)}_{S}$.\par 
	 In particular, the stratification of $G/B$ by Schubert cells $C_{w}$ is cellular.
\end{prop}
\begin{proof}
	As the construction of the Schubert cells is compatible with base change, we may assume without loss of generality that $S=\Spec(\ZZ)$. Then the proposition follows from \cite[II. 13]{Jantzen}
\end{proof}

 \subsection{Equivariant motives of linear Artin stacks}
 \label{sec.main}
 In the following, we assume that $S$ is an affine scheme and $H$ an $S$-group scheme. Further, every scheme will be qcqs of finite type over $S$. \par We fix an $E_{\infty}$-ring spectrum $M_{\B{H}}\in\SH(\B{H})$ that satisfies Assumption \ref{ass.vanish}.

\begin{defi}
\label{defi.lin}
	A \textit{linear $S$-scheme $(X,(X_{n})_{n\geq 0})$} consists of an $S$-scheme $X$ and a filtration by closed subschemes
$$
	\emptyset = X_{-1} \hookrightarrow X_{0}\hookrightarrow X_{1}\hookrightarrow \dots \hookrightarrow X_{n}\hookrightarrow \dots \hookrightarrow X
$$ 
such that each $X_{n-1}\rightarrow X_{n}$ is a closed immersion, each $X_{n}\setminus X_{n-1}$ is isomorphic to a coproduct vector bundles over $S$ and the natural closed immersion $\colim_{n} X_{n} \hookrightarrow X$ is an isomorphism on the reduced loci.\par 
If $X_{n}\setminus X_{n-1}$ is isomorphic to a coproduct of affine spaces over $S$, we call $(X,(X_{n})_{n\geq 0})$ \textit{affinely linear}.
\end{defi}

\begin{defi}
	Let $(X,(X_{n})_{n\geq 0})$ be a linear $S$-scheme such that $X$ admits a $H$-action. We say that $(X,(X_{n})_{n\geq 0})$ is \textit{$H$-equivariant} each of the $X_{n}$ is stabilized by $H$.
\end{defi}

\begin{rem}
\label{rem.eq.quot}
	 Let $(X,(X_{n})_{n\geq 0})$ be a $H$-equivariant linear $S$-scheme and assume that $X$ is quasi-compact. Further, let us set $U_{n}\coloneqq X_{n}\setminus X_{n-1} = \coprod_{j\in J_{n}} V(\Ecal_{n,j})$. By $H$-invariance, we get an action of $H$ on $U_{n}$. In particular, we can take the associated quotient stack $U_{n}/H$. As explained in Example \ref{ex.motive.of.vb} this yields for each $j\in J_{n}$ a finite locally free $H$-equivariant sheaf $\Ecal_{n,j,H}$, such $U_{n}/H\cong \coprod_{j\in J_{n}} V(\Ecal_{n,j,H})$ together with a finite set $I_{n,j}\subseteq \NN_{0}$ and a decomposition $S=\coprod_{j\in J_{n}}\coprod_{i\in I_{n,j}}S_{i,j}$ such that $\Ecal_{n,j,H|S_{i,j}}$ is finite locally free of rank $i$.
\end{rem}

\begin{defi}
	A linear $S$-scheme $(X,(X_{n})_{n\geq 0})$ is called \textit{strict} for all $n\geq 0$, we have that $$\max\lbrace i\in \bigcup_{j\in J_{n-1}} I_{n-1,j}\rbrace<\min\lbrace i\in \bigcup_{j\in J_{n}} I_{n,j}\rbrace,$$ where the notation is as in Remark \ref{rem.eq.quot}.
\end{defi}

\begin{example}
\label{ex.flag.lin}
	Let us give the example, that motivates the definitions above. Let $G$ be a split reductive $S$-group scheme with maximal split torus $T$ contained in a Borel subgroup $B$. By Proposition \ref{lem.GB.cellular} the Schubert cells $C_{w}$ of $G/B$ are isomorphic to $\AA^{l(w)}_{S}$. Let us set the closed subscheme $X_{n}$ as the schematic image of $\coprod_{l(w)\leq n} C_{w}$ inside $G/B$. This yields a linear structure on $G/B$ by 
	$$
		X_{0}\subseteq X_{1}\subseteq \dots\subseteq G/B,
	$$  
	where $X_{n}\setminus X_{n-1}\cong \coprod_{l(w)=n} C_{w}$. By construction, each of the $X_{i}$ are $T$-invariant and further, the linear structure $(G/B,(X_{n})_{n\in \NN_{0}})$ is strict. Thus, this construction yields a strict affinely linear $T$-equivariant structure on $G/B$.
\end{example}

From now on, we assume that $H$ is a nice $S$-group scheme.

\begin{thm}
\label{thm.motive.lin.quot}
	Let $(X,(X_{n})_{n\geq 0})$ be a $H$-equivariant linear $S$-scheme such that $X$ is proper over $S$. Further, let us set $U_{n}\coloneqq X_{n}\setminus X_{n-1} = \coprod_{j\in J_{n}}V(\Ecal_{n,j})$. Then the following equivalences
	$$
		M_{\B{H}}(\quot{X/H})\simeq \bigoplus_{n\geq 0} M^{c}_{\B{H}}(\quot{U_{n}/H}) = \bigoplus_{n\geq 0}\bigoplus_{j\in J_{n}}\bigoplus_{i\in I_{n,j}} M_{\B{H}}\langle i\rangle,
	$$
	hold, where the notation is as in Remark \ref{rem.eq.quot}, if
	\begin{enumerate}
		\item[(i)] $M_{\B{H}}$ admits an orientation and the linear structure above is strict, or
		\item[(ii)] the linear structure above is strict affinely linear.
	\end{enumerate}\par
	Further, if $M_{\B{H}}$ satisfies Assumption \ref{ass.vanish.strong}, then we can omit the strictness in (i) and (ii).
\end{thm}
\begin{proof}
We will prove the theorem under the assumption (i). The proof under the other assumptions will follow easily by the same arguments.
	By definition, $X$ admits a filtration 
	$$
		X_{0}\hookrightarrow X_{1}\hookrightarrow \dots \hookrightarrow X,
	$$
	such that each $X_{i-1}\rightarrow X_{i}$ is a closed immersion with complement given by $U_{i}$. For simplicity, we will assume that every quotient in the following is taken with respect to the \'etale topology.  By $H$-equivariance we may assume that each $X_{i}$ is stabilized by $H$. We see that $X/H$ admits a filtration 
	$$
		X_{0}/H\hookrightarrow X_{1}/H\hookrightarrow \dots \hookrightarrow X/H,
	$$
	where each of the $X_{n-1}/H\hookrightarrow X_{n}/H$ is a closed immersion with complement given by $U_{n}/H = \coprod_{j\in J_{n}}V(\Ecal_{n,j,H})$. 
	As the $X_{n}$ are proper, the $X_{n}/H$ are also proper over $\B{H}$ and therefore $M_{c,\B{H}}(X_{n}/H)\simeq M_{\B{H}}(X_{n}/H)$ by definition. In particular the localization sequence for motives with compact support (cf. Lemma \ref{lem.loc.for.stacks}) yields the fiber diagram
	$$
		M_{\B{H}}(X_{n-1}/H)\rightarrow M_{\B{H}}(X_{n}/H)\rightarrow M^{c}_{\B{H}}(U_{n}/H).
	$$
	We claim that this sequence splits.\par 
	Indeed, as explained in Remark \ref{rem.eq.quot}, we have $$M_{c,\B{H}}(\quot{U_{n}/H}) \simeq \bigoplus_{j\in J_{n}}\bigoplus_{i_{n}\in I_{n,j}} M_{\B{H}}\langle i_{n}\rangle .$$\par 
	By induction we may assume that $$M_{\B{H}}(X_{n-1})\simeq\bigoplus_{k=0}^{n-1}\bigoplus_{j_{k}\in J_{k}} \bigoplus_{i_{k}\in I_{k,j_{k}}} M_{\B{H}}\langle i_{k}\rangle .$$
	 In particular, any morphism  $\delta\colon M_{c,\B{H}}(U_{n}/H)\rightarrow M_{\B{H}}(X_{n-1}/H)[1]$ corresponds to an element in 
	$$
		\prod_{k=0}^{n-1}\prod_{j_{k}\in J_{k}}\prod_{i_{k}\in I_{k,j_{k}}}\prod_{j\in J_{n}}\prod_{i_{n}\in I_{n,j}}\Hom_{\SH(\B{H})}(M_{\B{H}}, M_{\B{H}}\langle i_{k}-i_{n}\rangle[1]).
	$$
	As $i_{k}-i_{n} < 0$ for any $i_{k}\in I_{k,j_{k}}$ with $0\leq k\leq n-1$ and $i_{n}\in I_{n,j}$ by strictness, we see using Assumption \ref{ass.vanish} that $\delta =0$.\par 
	Now induction over $n$ concludes the proof of the first assertion using that the motives only depend on the underlying reduced structure.
\end{proof}

\begin{cor}
\label{cor.T-Flag}
	Let $G$ be a split reductive $S$-group scheme with maximal split torus $T$ that is contained in a Borel subgroup $B$ of $G$. Then 
	$$
		M_{\B{T}}(\quot{T\bs G/B})\simeq \bigoplus_{w\in W} M_{\B{T}}\langle l(w)\rangle .
	$$
\end{cor}
\begin{proof}
	This immediately follows with Theorem \ref{thm.motive.lin.quot} and Example \ref{ex.flag.lin}.
\end{proof}

\subsection{Application to $T$-equivariant cohomology theories of flag varieties}
In this subsection, we will prove Theorem \ref{int.thm.2}, Corollaries \ref{int.cor.1} and \ref{int.cor.2} and Proposition \ref{int.prop.1}.\par
So, let $S$ be a Noetherian regular affine scheme. In the following $G$ is a split reductive $S$-group scheme with maximal split torus $T$ contained in a Borel subgroup $B$ of $G$.

\subsubsection{Integral equivariant $K$-theory}
\label{sec.int.K}
Let $\KGL_{\B{T}}$ be the $K$-theory spectrum computing homotopy invariant $K$-theory for smooth representable stacks over $\B{T}$ (cf. Theorem \ref{thm.6.ff} (xii)). Note that $\KGL_{\B{T}}$ satisfies Assumption \ref{ass.vanish} by Example \ref{ex.vanish}. Then Corollary \ref{cor.T-Flag} yields the following computation.

\begin{cor}
\label{cor.KH.flag}
Let $G$ be a split reductive $S$-group scheme with maximal split torus $T$ that is contained in a Borel subgroup $B$ of $G$. Then
	$$\textup{KH}(\quot{T\bs G/B}) \simeq \bigoplus_{w\in W} \KH(\B{T}).$$
\end{cor}
\begin{proof}
	This follows from Bott periodicity\footnote{The Bott periodicity yields $\KGL_{X}\langle \Ecal\rangle\simeq\KGL_{X}$ for any finite locally free sheaf $\Ecal$ over a scalloped stack $X$.} of $\KGL$, Corollary \ref{cor.T-Flag} and Theorem \ref{thm.6.ff}.
\end{proof}

By construction $\quot{T\bs G/B}$ and $\B{T}$ are quotients of smooth Noetherian $S$-schemes by a nice $S$-group scheme. In particular, we see that $\KH(\quot{T\bs G/B})$ and $\KH(\B{T})$ are connective and their homotopy groups are computed by genuine equivariant $K$-theory (cf. \cite[Thm. 5.7]{HoyKri}). Hence, for any $i\in \ZZ$ we have 
\begin{equation}
\label{eq.higher.K}
	K^{T}_{i}(G/B) = \bigoplus_{w\in W} K^{T}_{i}(S).
\end{equation}
	
For rational $K_{0}$ this is nothing new over a field, as the Schubert classes yield an $R(T)_{\QQ}$-basis of $K^{T}_{0}(G/B)_{\QQ}$ (cf. \cite{KostKum}).\par 
Let $S=\Spec(k)$. With rational coefficients, we can use \cite[Prop. A.5]{KriRR} to get isomorphisms of $R(T)$-modules
\begin{equation}
	\label{eq.Ki.flag}
		K^{T}_{i}(G/B)_{\QQ}\cong\bigoplus_{w\in W} K_{i}(k)_{\QQ} \otimes_{\QQ} R(T)_{\QQ} \cong K_{i}(k)_{\QQ}\otimes_{\QQ} K^{T}_{0}(G/B)_{\QQ}.
\end{equation}

\subsubsection{Completed equivariant $K$-theory}
\label{sec.gen.motive}

In the following, we want to extend our results to other formalisms of motives. To look at these different definitions all at once, we use the formalism of a motivic $\infty$-category $\Dbf$ with full six-functor formalism on scalloped stacks (cf. \cite[Prop. 5.13]{KR1}). We do not want to give an explicit definition of a motivic $\infty$-category, as it boils down to rewriting the axioms of the $6$-functor formalism. But we want to give $2$ examples that are of interest to us.

\begin{example}
\label{ex.D}
The following $2$ examples are motivic $\infty$-categories with a full $6$-functor formalism on scalloped stacks.
\begin{enumerate}
		\item[(1)] Let $X$ be an Artin stack over $S$.  Let us now further assume that $X$ is quasi-separated with quasi-separated representable diagonal and has a smooth cover that admits Nisnevich locally sections (we call such covers \textit{smooth-Nisnevich}), e.g. any quasi-separated algebraic stacks with separated diagonal admits a smooth-Nisnevich cover (cf. \cite[Thm 1.2 (1)]{Desh}). The lisse-extended stable homotopy category $\Dbf = \SH_{\lhd}$ is defined via
		$$
			\SH_{\lhd}(X)\coloneqq \colim_{(T,t)} \SH(T),
		$$
		where limit is taken in the $\infty$-category of pairs $(T,t)$, where $t\colon T\rightarrow X$ is a smooth morphism and $T$ an algebraic space (cf. \cite[\S 12]{KR1}). The $\infty$-category $\SH_{\lhd}(X)$ is equivalent to the right Kan extension of $\SH$ to smooth-Nisnevich stacks, evaluated at $X$ (cf. \cite[Cor. 12.28.]{KR1}). This proves the extension of a full six-functor formalism, homotopy invariance and existence of a localization sequence to Artin stacks for this $\infty$-category (this follows from the arguments of \cite[App. A]{Khan1}). Note that for the existence of $\sharp$-pushforward and $!$-formalism there is no need for representability in this context. Further, they show that when $S$ is the spectrum of a perfect field and $X$ is the quotient of a smooth scheme by an algebraic group, the motivic cohomology spectrum $\Hom_{\SH_{\lhd}(X)}(1_{X},M)$ for $M\in \SH_{\lhd}(X)$ can be computed by a Borel construction - this is precisely the construction Edidin-Graham make to define equivariant Chow groups as seen in Example \ref{ex.comp} - (cf. \cite[Thm. 12.16]{KR1}).
		\item[(2)] Let us consider the \'etale localized rational stable homotopy category $\SH_{\QQ,\et}$. This $\infty$-category can be right Kan extended to Artin stacks over $S$. As by definition $\SH_{\QQ,\et}$ satisfies \'etale descent, we can extend the $6$-functor formalism, homotopy invariance and the localization sequence to $\SH_{\QQ,\et}$ on Artin stacks (cf. \cite[App. A]{Khan1}).
		Again, there is no need for representability for the existence of a $6$-functor formalism.
	\end{enumerate}
\end{example}

The universal property of the stable homotopy category yields a unique system of comparison maps $R_{X}\colon\SH(X)\rightarrow\Dbf(X)$ (cf. \cite[Prop. 5.13]{KR1}) for any motivic $\infty$-category $\Dbf(X)$. The family of functors $R$ is compatible with $\sharp$-pushforward, $*$-inverse image, Thom twists and tensor products.\par 
As the family of functors $R_{X}$ is monoidal, we can extend this to modules over any $E_{\infty}$-ring spectrum in $\SH(X)$, i.e. if $M_{X}\in \SH(X)$ is an $E_{\infty}$-ring spectrum, then there is a functor
$$
	R_{M}(X)\colon \SH(X)_{M}\rightarrow R_{X}(M)\textup{-Mod}(\Dbf(X)) 
$$
compatible with pullbacks of $M$ and thus also $\sharp$-pushforward, $*$-inverse image, Thom twists and tensor products in module spectra.\par
Let us relate this construction, to our examples above by applying it to motivic cohomology.

\begin{example}\par
Let $X$ be an Artin stack over $S$ and let $M\ZZ\in\SH(X)$ be the motivic cohomology ring spectrum.
	\begin{enumerate}
		\item[(1)] Let us now further assume that $X$ admits a smooth-Nisnevich cover (cf. Example \ref{ex.D} (1)). Let $M\ZZ^{\lhd}$ be the image of $M\ZZ$ under $R_{X}\colon \SH(X)\rightarrow \SH_{\lhd}(X)$. \par The $\infty$-category $\SH_{\lhd}(X)$ is equivalent to the right Kan extension of $\SH$ to smooth-Nisnevich stacks, evaluated at $X$ (cf. Example \ref{ex.D} (1)). Using this construction, the $\infty$-category of $M_{X}^{\lhd}$-modules in $\SH_{\lhd}(X)$ can be describe as follows. \par 
		We right Kan extends Spitzweck motives $\DM$ to prestacks along Nisnevich covers (cf. \cite{CVDS}). In this way, one can construct the exceptional pullback and pushforward for finite type morphisms. Further, for prestacks given by a quotient $X/G$ of a scheme by a group scheme, the $\infty$-category $\DM(X/G)$ is equivalent to $\lim_{\Delta}\DM(\Bres^{\bullet}(X,G))$, where $\Bres^{\bullet}(X,G))$ denotes the Bar resolution of $X$ with respect to $G$. If $G$ is special (e.g. $G=T$), then any \'etale $G$-torsor is Zariski locally trivial and in particular the \'etale sheafification of $X/G$, which we usually denoted by $\quot{X/G}$, agrees with the Nisnevich sheafification. As seen in \textit{op.cit.} this allows one to compute $\DM(\quot{X/G})=\DM(X/G) = \lim_{\Delta}\DM(\Bres^{\bullet}(X,G))$ (here $\DM(X/G)$ denotes the evaluation of $\DM$ at the presheaf quotient $X/G$).
		\item[(2)] Let us assume that $S$ is of finite type over an excellent Noetherian scheme of dimension $\leq 1$.  Further, let us denote the image of $M\ZZ$ in $\SH(X)_{\QQ,\et}$ with $M\QQ$. Then $M\QQ$ can be glued from the Beilinson motivic cohomology spectrum along a smooth cover $X$ and we can describe the $M\QQ$-modules in $\SH(X)_{\QQ,\et}$ as follows.\par
		The right Kan extension of Beilinson motives $\DM_{\QQ}$ to Artin stacks admits an extension of the full six-functor formalism (cf. \cite{RS1}). This is achieved by gluing along smooth covers. There are different ways to see this, but we do not want to go into details and refer to the paragraph before \cite[\S A.2]{Khan1}). 
	\end{enumerate}
\end{example}

\begin{rem}
	An important example in our context, that fits into the setting of Example \ref{ex.D} (1), are quotient stacks of the form $\quot{X/H}$, where $X$ is a quasi-separated scheme and $H$ is a nice group scheme (cf. \cite[Rem. 12.24]{KR1}).
\end{rem}

From now on let us assume that the six-functor formalism in $\Dbf$ exists for not necessarily representable morphisms.

\begin{notation}
In the following we fix an $E_{\infty}$-ring spectrum $M_{S}\in\SH(S)$ we will denote its image in $\Dbf(S)$ with $M_{S}^{\Dbf}$. Again, we will denote the pullback of $M_{S}^{\Dbf}$ under a map $X\rightarrow S$, where $X$ is a scalloped stack, with $M_{X}^{\Dbf}$. Further, we will denote the $\infty$-category of $M_{X}^{\Dbf}$-modules in $\Dbf(X)$ with $\Dbf(X)_{M}$.\par
	 For any morphism scalloped stacks $f\colon X\rightarrow Y$  of finite type, we will denote the motives in $\Dbf$ with $M^{\Dbf}_{Y}(X)\coloneqq f_{!}f^{!}M^{\Dbf}_{X}$ and $M^{c,\Dbf}_{Y}\coloneqq f_{*}f^{!}M^{\Dbf}_{X}$.\par 
	We also define \textit{motivic cohomology of a scalloped stack $X$ with coefficients in $M^{\Dbf}$} as 
	$$	
		H^{n,m}_{\Dbf}(X,M)\coloneqq \Hom_{\Dbf(X)}(1_{X},M^{\Dbf}_{X}(n)[m]).
	$$
\end{notation}

\begin{rem}
\label{rem.classical}
Assume $f\colon X\rightarrow S$ is a smooth scalloped stack over $S$. Then $f_{!}f^{!}M_{S}\simeq f_{\sharp}M_{X}$ and we see that 
$$
	H^{n,m}_{\Dbf}(X,M) \simeq \Hom_{\Dbf(X)_M}(M_{S}^{\Dbf}(X),M_{S}^{\Dbf}( n) [m]).
$$\par
	If $X$ is smooth over $\B{H}$ for some nice group scheme $H$, we therefore can transport all of the results of Section \ref{sec.main} to $M_{S}(X)$ via $\sharp$-pushforward along the structure map $\B{H}\rightarrow S$.
\end{rem}

Working over the base $S$, we can analyze the motive of strict linear schemes. This is classical, and the proof is achieved by \textit{mutas mutandis} of the proof of Theorem \ref{thm.motive.lin.quot}.

\begin{prop}
\label{prop.motive.lin}
Let us assume that $H^{i,1+2i}_{\Dbf}(S,M) = 0$ for all $i> 0$.
	Let $(X,(X_{n})_{n\geq 0})$ be a linear $S$-scheme such that $X$ is proper over $S$. Further, let us set $U_{n}\coloneqq X_{n}\setminus X_{n-1} =  \coprod_{j\in J_{n}}V(\Ecal_{n,j})$. Then the equivalences
	$$
		M^{\Dbf}_{S}(X)\simeq \bigoplus_{n\geq 0} M^{\Dbf,c}_{S}(U_{n}) = \bigoplus_{n\geq 0}\bigoplus_{j\in J_{n}}\bigoplus_{i\in I_{n,j}} M_{S}\langle i\rangle,
	$$
	hold, where the notation is as in Remark \ref{rem.eq.quot}, if
	\begin{enumerate}
		\item[(i)] $M_{S}^{\Dbf}$ admits an orientation and the linear structure above is strict, or
		\item[(ii)] the linear structure above is strict affinely linear.
	\end{enumerate}\par
	Further, if $H^{i,1+2i}_{\Dbf}(S,M) = 0$ for all $i\in \ZZ$, then we can omit the strictness in (i) and (ii).
\end{prop}
\begin{proof}
	This is completely analogous to the proof of Theorem \ref{thm.motive.lin.quot}.
\end{proof}

The structure of the motive of linear stacks allows us to rewrite Theorem \ref{thm.motive.lin.quot}.

\begin{cor}
\label{cor.lin.quot.tensor}
	Let $H$ be a nice $S$-group scheme. Let $(X,(X_{n})_{n\geq 0})$ be a $H$-equivariant strict linear $S$-scheme such that $X$ is smooth and proper over $S$. Then 
	$$
		M^{\Dbf}_{S}(\quot{X/H})\simeq M^{\Dbf}_{S}(X)\otimes M^{\Dbf}_{S}(\B{H}).
	$$
\end{cor}
\begin{proof}
	This follows immediately from Proposition \ref{prop.motive.lin} and Theorem \ref{thm.motive.lin.quot} via $\sharp$-pushforward along the structure map $\B{H}\rightarrow S$ (cf. Remark \ref{rem.classical}).
\end{proof}

\begin{example}
\label{ex.comp}
Let us validate Corollary \ref{cor.lin.quot.tensor} using a more direct computation over $S=\Spec(k)$, where $k$ is a field and restricting ourselves to the case $\Dbf=\SH_{\QQ,\et}$ and $M^{\Dbf}_{S} = M\QQ_{S}\in \Dbf(S)$ the rational motivic cohomology ring spectrum.\par 
 Let $G=\SL_{2,S}$ and $T=\Gm_{,S}$ the standard diagonal torus. Let $B$ be the Borel subgroup of upper triangular matrices in $G$. Then $G/B\cong \PP_{S}^{1}$ and the action of $T$ on $\PP_{S}^{1}$ induced by conjugation on $G/B$. In particular, the action of $T$ on $\PP_{S}^{1}$ is given by multiplication. The motive of $\quot{T\bs \PP^{1}_{k}}$ can be computed in the following way - analogous to the computation of its intersection ring (cf. \cite[\S 3.3]{EG}).\par
Let us fix an $i\in\ZZ$. Further, let $V$ be a non-trivial representation of $T$ over $k$. We denote with $\VV \coloneqq \Spec(\Sym(V))$ the associated vector bundle over $S$. Then we define the scheme $U_{i}$ for each $h\colon Q\rightarrow S$ via
$$
 	U_{i}(Q)\coloneqq \lbrace u\in \Hom(h^{*}\VV,h^{*}\VV^{i}) \mid \Coker(u)\textup{ is finite free of rank } i\rbrace
$$
Then $T$ acts freely on $U_{i}$ and one can show that the codimension of $U_{i}$ in the $k$-vector bundle $\VV_{i}\coloneqq \Spec(\Sym(V\otimes V^{\vee})^{i})$ is greater than $-i$. Further, we can see that 
$$
	\PP^{1}_{k}\times^{T}_{k} U_{i-1} \cong \PP(\Ocal_{\PP^{i-1}}(1) \oplus \Ocal_{\PP^{i-1}}(1)) \rightarrow \PP^{i-1}_{k} \cong U_{i-1}/T
$$
 is a $\PP^{1}_{k}$-bundle. Thus, the projective bundle formula yields
 $$
 	M_{S}^{\Dbf}(\PP^{1}_{k}\times^{T}_{k} U_{i-1}) \cong M_{S}^{\Dbf}(\PP^{i-1}_{k}) \oplus M_{S}^{\Dbf}(\PP^{i-1}_{k})\langle 1\rangle.
 $$
Now the motive of $M^{\Dbf}_{S}(\quot{T\bs G/B})$ is isomorphic to the colimit over $i$ of $M_{S}^{\Dbf}(\PP^{1}_{k}\times^{T}_{k} U_{i-1})$ (this can be followed from \cite[Prop. A.7]{HosLeh} resp. \cite[p. 2107]{Tot}). Therefore, we finally have
	$$
	 M^{\Dbf}_{S}(\quot{T\bs G/B})\simeq M_{S}^{\Dbf}(B\Gm_{,S})\oplus M_{S}^{\Dbf}(B\Gm_{,S})\langle 1\rangle\simeq M^{\Dbf}_{S}(\PP^{1}_{S})\otimes M^{\Dbf}_{S}(\Gm_{,S}),
	$$
	where the last equivalence follows again from the projective bundle formula.
\end{example}

\begin{rem}
	The above example and computations also hold for $\SH_{\lhd}$ and the integral motivic cohomology ring spectrum if either $k$ has characteristic $0$ or after inverting the characteristic of $k$ (cf. \cite[Thm. 12.16]{KR1}).
\end{rem}

Let us specialize to the case of $T$ acting on the flag variety $G/B$ and the case where $S=\Spec(k)$ is the spectrum of a field. Assume for this paragraph that $k$ has characteristic $0$. As mentioned in Example \ref{ex.D}, the motivic cohomology spectrum can be computed using the Borel construction in the cases that are of interest to us. For Chow groups of stacks, this gives the right computation. For $K$-theory this is no longer true. In fact, one can show that the Borel construction yields the completion of equivariant $K$-theory along the augmentation ideal (cf. \cite[Thm. 1.2]{Kri}), i.e.
$$
	 H^{0,-i}_{\SH_{\lhd}}(\quot{T\bs G/B},\KGL^{\lhd})\simeq K^{T}_{i}(G/B)^{\wedge_{I_{T}}},
$$
where $I_{T}\subseteq R(T)$ is the ideal generated by virtual rank $0$ representations and $K^{T}_{i}(G/B)^{\wedge_{I_{T}}}$ is the completion along $I_{T}K^{T}_{i}(G/B)$. Again the above stays true in characteristic $p>0$, after inverting $p$.

\begin{example}
\label{ex.K.comp}
	Let $G=\Gm$. Then $R(G) = \ZZ[T,T^{-1}]$ and the augmentation ideal $I_{G}$ is generated by $1-T$. Thus, $R(G)^{\wedge_{I_{G}}}= \ZZ[\![T]\!]$ and we see that indeed $R(G)$ is not $I_{G}$-complete. Therefore, the lisse-extended $K$-theory spectrum does not recover $K$-theory.
\end{example}

A similar result appears when one wants to prove an equivariant form of the Riemann-Roch Theorem (cf. \cite{EGRR}). The same holds, if we consider cohomology theories in $\SH_{\QQ,\et}$ as they satisfy \'etale descent. This descent property is the ambiguity here. One can show that even rational $K$-theory of stacks does not satisfy \'etale descent (for $G$-theory one can give precise conditions on quotient stacks, cf. \cite[\S 3]{Jo1}). Nevertheless, we want to show the implications of our calculations for Chow groups and completed $K$-theory.\par

The upshot of Corollary \ref{cor.lin.quot.tensor} is that it gives us a tensor description 
\begin{equation}
\label{eq.mot.flag}
	M_{S}^{\Dbf}(\quot{T\bs G/B}) = M^{\Dbf}_{S}(G/B)\otimes M^{\Dbf}_{S}(\B{T})
\end{equation}
and we want to use this to get a tensor description of completed $T$-equivariant $K$-theory of $G/B$. In this case, we would need a K\"unneth formula for $K$-theory. For equivariant $K_{0}$, in our special case, this is known and follows by the spectral sequence induced for example on $G$-theory (cf. \cite[Thm. 4.1]{Jo2}). For higher equivariant $K$-groups there is no K\"unneth formula, as this fails even for non-equivariant $K$-theory.

\begin{example}
	Let us consider $\AA^{1}_{k}\rightarrow \Spec(k)$ the projection of the affine line. Then $K_{1}(\AA^{2}_{k}) = K_{1}(\AA^{1}_{k}\times_{k} \AA^{1}_{k}) = k^{\times}$ by homotopy invariance, whereas $K_{1}(\AA^{1}_{k})\otimes_{\ZZ} K_{0}(\AA^{1}_{k}) \oplus K_{0}(\AA^{1}_{k})\otimes_{\ZZ} K_{1}(\AA^{1}_{k}) = k^{\times}\oplus k^{\times}$.
\end{example} 

Instead of a K\"unneth formula, one gets a spectral sequence for $K$-theory and higher Chow theory, at least when one of the factors comes from a linear scheme (cf. \cite[Thm. 4.3]{Jo2}). In fact, Totaro shows that Chow groups commute with tensor products if and only if one of the associated motives of the factors is Tate (relative over a field, cf. \cite[Thm. 7.2]{Tot}). Thus, if $X$ is a (smooth) strict linear $S$-scheme and $Y$ an arbitrary (smooth) scheme, we have $A^{*}(X\times_{S}Y) = A^{*}(X)\otimes_{\ZZ} A^{*}(Y)$. In \cite{Jo2} this is also a result of the associated spectral sequence of motivic cohomology of linear varieties. Dugger and Isaksen generalize this idea to arbitrary cellular motivic cohomology theories like motivic cohomology, algebraic $K$-theory and algebraic cobordism (cf. \cite{DI}). 

\begin{prop}[Tor spectral sequence]
\label{prop.tor}
Let $\Dbf = \SH_{\QQ,\et}$. Let $M_{S}$ be either $\KGL_{\QQ,S}$ or $M\QQ_{S}$ inside $\Dbf(S)$. Let $(X,(X_{n})_{n\geq 0})$ be a linear $S$-scheme such that $X$ is proper over $S$ and let $Y$ be a smooth algebraic stack. Then for each $n\in \ZZ$ there is a spectral sequence
	$$
		\textup{Tor}^{H^{*,n}_{\Dbf}(S,M)}_{p}({H^{*,n}_{\Dbf}(X,M)},H^{*,n}_{\Dbf}(Y,M))_{q}\Rightarrow H^{p+q,n}_{\Dbf}(X\times_{S}Y,M).
	$$
\end{prop}
\begin{proof}
	By Proposition \ref{prop.motive.lin} the motive $M_{S}^{\Dbf}(X)$ is a direct sum of Tate-motives. Therefore, the result follows with \cite[Thm. 6.2, Thm 6.4, Prop. 7.7]{DI}.
\end{proof}
We can use Proposition \ref{prop.tor} to see that
\begin{equation}
\label{eq.K.flag}
	K_{0}^{T}(G/B)_{\QQ}^{\wedge I_{T}}\cong K_{0}(G/B)_{\QQ}\otimes_{\QQ} K^{T}_{0}(S)_{\QQ}^{\wedge_{I_{T}}}
\end{equation}
noting that the $K$-theory of $G/B$ and $\B{T}$ is connective (cf. Section \ref{sec.int.K}).\par
 The comparison of Beilinson motivic cohomology with higher Chow groups yield
$$
	H^{n,2n}_{\DM}(X,\QQ) = A^{n}(X)_{\QQ},
$$
as for $k<n$, we have $H^{k,n}_{\DM}(X,\QQ) = A^{n}(X,2k-2n)_{\QQ} = 0$. If $S=\Spec(k)$ is the spectrum of a field, we can use that equivariant Chow groups are given via the Borel construction, we see as before that 
\begin{equation}
\label{eq.Chow.flag}
	A^{*}_{T}(G/B)_{\QQ}\cong A^{*}(G/B)_{\QQ} \otimes_{\QQ} A_{T}^{*}(S)_{\QQ}.
\end{equation}

\begin{rem}
	It should not be difficult to generalize Proposition \ref{prop.tor} to the case of Spitzweck motives and get an integral version of the above results for completed $K$-theory and Chow theory, at least after inverting the characteristic of the ground field (in the positive characteristic setting). But as this probably boils down to just rewriting the results of Dugger and Isaksen, we did not follow this further and leave it to the reader.
\end{rem}

\addcontentsline{toc}{section}{References}
\bibliographystyle{halpha-abbrv}
\bibliography{T-motives}

\end{document}